\documentclass[11pt]{amsart}
\usepackage{amsmath, amssymb, amsbsy, amsfonts, amsthm, latexsym, amsopn, amstext, amsxtra, euscript, amscd, color, mathrsfs}
\usepackage[normalem]{ulem}
\usepackage{soul}
\usepackage{cite}

\PassOptionsToPackage{hyphens}{url}\usepackage{hyperref}
 
 \usepackage[capbesideposition=outside,capbesidesep=quad]{floatrow}

\restylefloat{table}
\restylefloat{table}
            
\usepackage{multirow,caption}
            
\usepackage{amscd}
\usepackage{color,enumerate}
\newcommand{\RNum}[1]{\lowercase\expandafter{\romannumeral #1\relax}}

\usepackage[colorinlistoftodos,prependcaption,textsize=tiny]{todonotes}

\newtheorem{thm}{Theorem}[section]
\newtheorem{lem}[thm]{Lemma}
\newtheorem{cor}[thm]{Corollary}

\newtheorem{rem}[thm]{Remark}

\newtheorem{thm-con}[thm]{Theorem-Conjecture}
\numberwithin{equation}{section}

\def\cN{{\mathcal N}}

\theoremstyle{definition}

\newcommand{\F}{\mathbb F}

\def\Trn{{\rm Tr}_1^n}

\newcommand{\cardinality}[1]{\# #1}

\usepackage{todonotes}
\usepackage{lipsum}
\usepackage{marginnote}

\begin{document}
\title[APN functions in odd characteristic]{\Large On APN functions in odd characteristic, the disproof of a conjecture and related problems}
\author{Daniele Bartoli}
\address{Department of Mathematics and Computer Science, University of Perugia,  06123 Perugia, Italy}
\email{daniele.bartoli@unipg.it}
\author{Pantelimon St\u anic\u a}
\address{Applied Mathematics Department, Naval Postgraduate School,
Monterey, CA 93943, USA}
\email{pstanica@nps.edu}

\date{\today}


\begin{abstract}
In this paper disprove a conjecture by Pal and Budaghyan (DCC, 2024) on the existence of a family of APN permutations, but showing that if the field's cardinality $q$ is larger than~$9587$, then those functions will never be APN. Moreover, we discuss other connected families of functions, for potential APN functions, but we show that they are not good candidates for APNess if the underlying field is large, in spite of the fact that they  though they are APN for small environments.
\end{abstract}

\maketitle
 
\section{Introduction}

Let $\F_q$ be the finite field with $q=p^k$ elements, where $k$ is a positive integer. We denote by $\F_q^*$ the multiplicative group of nonzero elements of $\F_q$ and by $\F_q[X]$ the polynomial ring in the indeterminate $X$ over a finite field $\F_q$. A polynomial $f \in \F_q[X]$ is called a permutation polynomial if the equation $f(X)=a$ has exactly one solution in $\F_q$ for each $a\in \F_q$.  Below, we let $\chi_1(a)=\exp\left(\frac{2\pi i\Trn(a)}{q}\right)$ be the principal additive character of $\F_q$, $q=p^n$.

Given a vectorial $p$-ary  function $f:\F_{p^n}\to\F_{p^n}$, the derivative of $f$ with respect to~$a \in \F_{p^n}$ is the $p$-ary  function
$
 D_{a}f(x) =  f(x + a)- f(x), \mbox{ for  all }  x \in \F_{p^n}.
$ 
For an $(n,m)$-function  
$F$, and $a\in\F_{p^n},b\in\F_{p^m}$, we let $\Delta_F(a,b)=\cardinality{\{x\in\F_{p^n} : F(x+a)-F(x)=b\}}$. We call the quantity
$\delta_F=\max\{\Delta_F(a,b)\,:\, a,b\in \F_{p^n}, a\neq 0 \}$ the {\em differential uniformity} of $F$. If $\delta_F\leq \delta$, then we say that $F$ is differentially $\delta$-uniform. If $m=n$ and $\delta=1$, then $F$ is called a {\em perfect nonlinear} ({\em PN}) function, or {\em planar} function. If  $m=n$ and $\delta=2$, then $F$ is called an {\em almost perfect nonlinear} ({\em APN}) function. It is well known that PN functions do not exist if $p=2$. 

We will denote by $\eta(\alpha)$ the quadratic character of $\alpha$ (that is, $\eta(\alpha)=0$ if $\alpha=0$, $\eta(\alpha)=1$ if $0\neq \alpha$ is a square, $\eta(\alpha)=-1$ if $\alpha$ is not a square).


\section{Preliminaries from Function Field Theory} \label{Sec:AlgGeo}

In this paper we will make use of some concepts concerning Function Field Theory. This will yield  a lower bound on the number of permutation polynomials of the desired shape. 
 
We recall that a {\em function field} over a perfect field $\mathbb L$ is an extension $\mathbb F$ of $\mathbb L$ such that  $\mathbb F$ is  a finite algebraic extension of $\mathbb L(\alpha)$, with $\alpha$ transcendental over $\mathbb L$. For basic definitions on  function fields we refer to \cite{Sti}. In particular, the (full) constant field of $\mathbb F$ is the set of elements of $\mathbb F$ that are algebraic over $\mathbb L$.

If $\F^{\prime}$ is a finite extension of $\F$, then a place $P^{\prime}$ of $\F^\prime$ is said to be \emph{lying over} a place $P$ of $\F$ if $P \subset P^\prime$. This holds precisely when
$P=P^{\prime} \cap \F$. In this paper, $e(P^{\prime}|P)$ will denote the ramification index of $P^{\prime}$ over $P$. 
A finite extension $\F^{\prime}$ of a function field $\F$ is said to be \emph{unramified} if $e(P^{\prime}|P)=1$ for
every $P^{\prime}$ place of $\F^{\prime}$ and every $P$ place of $\F$ with $P^{\prime}$ lying over $P$. Since it is not needed here, we do not go into the tamely or totally ramification extensions' notions. Throughout the paper, we will refer to the following results. 

\begin{thm}\label{Th_Kummer}\textup{\cite[Cor. 3.7.4]{Sti}}
Consider an algebraic function field $\F$ with constant field $\mathbb L$ containing a primitive $n$-th root of unity ($n>1$ and $n$ relatively prime to the characteristic of $\mathbb L$). Let $u\in \F$ be such that 
there is a place $Q$ of $\F$ with $\gcd(v_Q(u),n)=1$ (see~\textup{\cite[Definition 1.1.2]{Sti}} for the definition of the discrete valuation function $v_Q$). Let $\F^{\prime}=\F(y)$ with $y^n=u$. Then:
\begin{enumerate}
\item[$(1)$] $\Phi(T)=T^n-u$ is the minimal polynomial of $y$ over $\F$. The extension $\F^{\prime}:\F$ is Galois of degree $n$ and the Galois group of  $\F^{\prime}:\F$ is cyclic;
\item[$(2)$] We have $$e(P^{\prime}|P)=\frac{n}{r_P} \quad \textrm{where} \quad r_P:= GCD(n,v_P(u))>0\,;$$
\item[$(3)$] $\mathbb L$ is the  constant field of $\F^\prime$;
\item[$(4)$] If $g^{\prime}$ (resp., $g$) is the genus of $\F^{\prime}$ (resp. $\F$), then
$$ g^{\prime}=1+n(g-1)+\frac{1}{2}\sum_{P\in\mathbb{P}(\F)}(n-r_P)\deg P\,. $$
\end{enumerate}
\end{thm}
An extension such as $\F^{\prime}$ in Theorem \ref{Th_Kummer}  is said to be a Kummer extension of $\F$.

Denote by $\F_q$ the finite field with $q$ elements and let $\mathbb{K}$ be the algebraic closure of  $\F_q$. A curve $\mathcal{C}$ in some affine or projective space over $\mathbb{K}$ is said to be defined over $\mathbb{F}_q$ if the ideal of $\mathcal{C}$ is generated by polynomials with coefficients in $\mathbb{F}_q$. Let $\mathbb{K}(\mathcal{C})$ denote the function field of $\mathcal C$. The subfield  $\mathbb{F}_q(\mathcal{C})$ of  $\mathbb{K}(\mathcal{C})$ consists of the rational functions on $\mathcal C$ defined over $\mathbb{F}_q$. The extension $\mathbb{K}(\mathcal C):\F_q(\mathcal C)$ is a constant field extension (see \cite[Section 3.6]{Sti}). In particular, $\mathbb{F}_q$-rational places of $\F_q(\mathcal C)$ can be viewed as the restrictions to $\F_q(\mathcal C)$ of  places of $\mathbb{K}(\mathcal{C})$ that are fixed by the Frobenius map on $\mathbb{K}(\mathcal{C})$. The center of an $\mathbb{F}_q$-rational place is an $\mathbb{F}_q$-rational point of $\mathcal{C}$; conversely,  if $P$ is a simple $\mathbb{F}_q$-rational point of $\mathcal{C}$, then the only place centered at $P$ is $\mathbb{F}_q$-rational. Through the paper, we sometimes use concepts from both Function Field Theory and Algebraic Curves. Concepts such as the valuation of a function at a place can be also seen as multiplicity of intersections of fixed algebraic curves; see \cite{Sti}. 

We now recall the well-known Hasse-Weil bound.

\begin{thm} \label{HasseWeil}\emph{(Hasse-Weil bound, \cite[Theorem 5.2.3]{Sti})}
The number $N_q$ of $\mathbb{F}_q$-rational places of a function field $\,\F$ with constant field $\F_q$ and genus $g$ satisfies 

$$|N_q - (q + 1)| \leq 2g\sqrt{q}.$$
\end{thm}

In order to apply the Hasse-Weil bound, the following lemma will be useful.

\begin{lem}\label{ConstantField}\textup{\cite[Lemma 1]{BGZ2016}}
Let $\F_q(\beta_1,\ldots,\beta_n)$ be a function field 
with constant field $\F_q$. Suppose that $f\in\F_q(\beta_1,\ldots,\beta_n)[T]$ is a polynomial which is irreducible over $\mathbb{K}(\beta_1,\ldots,\beta_n)[T]$. Then, for a root $z$ of $f$, the field $\F_q$ is the  constant field of $\F_q(\beta_1,\ldots,\beta_n)(z)$.
\end{lem}

 \section{A ``potential'' infinite class of APN functions}

  First, we prove the following result, which finds some low differential uniformity functions in odd characteristic. In our follow up result, we complete the proof for all the other cases and show that the conjecture of~\cite{BP24} is  false.
 \begin{thm}
Let $F(x)=x^{\frac{p^n+3}{2}}+u\,x^2$ on $\F_{p^n}$, where $u\in\F_{p^n}$ satisfies $u\notin\{0,\pm 1\}$.
If $u=-3$ and  $p^n\equiv 5\pmod 8$,  then the differential uniformity of $F$   on $\F_{p^n}$ is~$\leq 4$.  
 \end{thm}
\begin{proof}
For given $a\in\F_{p^n}^*, b\in\F_{p^n}$ we need to look at the differential equation $F(x+a)-F(x)=b$, that is,
\[
(x+a)^{\frac{p^n+3}{2}}+u(x+a)^2-x^{\frac{p^n+3}{2}}-ux^2=b.
\]
Denoting $t_a=\eta(x+a),t_x=\eta(x)$, and noting that $\frac{p^n+3}{2}=\frac{p^n-1}{2}+2$, the equation above becomes
\begin{equation}
\label{eq:main1}
 (x+a)^2 (t_a+u)-x^2(t_x+u)=b.
\end{equation}

We now distinguish four cases, which are displayed in the next table.
{\small
\begin{equation}
\label{Table:cases2}
\begin{array}{|l||r|r|c|c|c|}
\hline \text { Case } & t_a & t_x & \text { Equation \eqref{eq:main1} } & x & x+a \\
\hline 
C_{1,1} & 1 & 1 & a^2(u+1)+2a(u+1) x=b & \frac{b-(u+1)a^2}{2 a (u+1)} & \frac{b+(u+1)a^2}{2 a (u+1)} \\
\hline 
C_{-1,-1} & -1 & -1 &  a^2 (u-1) + 2 a (u-1) x=b & \frac{b-(u-1)a^2}{2 a (u-1)} & \frac{b+ (u-1)a^2}{2 a (u-1)} \\
\hline 
 {C_{-1,1}} &  {-1} &  {1} & (u-1)a^2 + 2a (u-1)  x - 2 x^2 =b&  \frac{a (u-1)\pm\sqrt{a^2 \left(u^2-1\right)-2b}}{2} &\frac{\left(a(u+1)\pm \sqrt{a^2 (u^2-1)-2 b}\right)}{2}  \\
\hline 
 {C_{1,-1}} & {1} &  {-1} &  (u+1) a^2  + 2a(u+1)  x + 2 x^2 =b&  \frac{  -a
   (u+1)\pm \sqrt{a^2 \left(u^2-1\right)+2 b} }{2} & \frac{  -a(u-1)\pm \sqrt{a^2 (u^2-1)+2 b}}{2} \\
\hline
\end{array}
\end{equation}
}

For easy referral, we label the putative solutions as $x_1$ (Case $C_{1,1}$), $x_2$ (Case $C_{-1,-1}$), $x_3,x_4$ (Case $C_{1,-1}$), $x_5,x_6$ (Case $C_{-1,1}$).

In Case $C_{1,1}$ we must have 
$\eta\left(\frac{b}{2a(u+1)}-\frac{a}{2} \right)=\eta\left(\frac{b}{2a(u+1)}+\frac{a}{2} \right)=1$, or equivalently,
 \begin{equation}
 \label{eq:case1}
\eta\left(\frac{b}{a^2(u+1)}-1 \right)=\eta\left(\frac{b}{a^2(u+1)}+1  \right)=\eta(2a).
\end{equation}
For  Case $C_{-1,-1}$ we need  $\eta\left(\frac{b}{2a(u-1)}-\frac{a}{2} \right)=\eta\left(\frac{b}{2a(u-1)}+\frac{a}{2} \right)=-1$, or equivalently,
 \begin{equation}
 \label{eq:case2}
\eta\left(\frac{b}{a^2(u-1)}-1 \right)=\eta\left(\frac{b}{a^2(u-1)}+1  \right)=-\eta(2a).
 \end{equation}
 In Case $C_{-1,1}$, for at least a solution to exist, one needs the expression inside the root to be a square, and further $\eta(x_3x_4)=\eta((x_3+a)(x_4+a))=1$, so, 
 $\eta\left(-2b+a^2 \left(u^2-1\right)\right)= \eta\left(b-(u-1)a^2\right)=\eta\left(b+(u+1)a^2\right)=1$, or equivalently,
 \begin{equation}
 \label{eq:case3}
 \begin{split}
 \eta\left(\frac{-2b}{a^2 \left(u^2-1\right)}+1\right)&= \eta(u^2-1),\\
 \eta\left(\frac{b}{(u-1)a^2}-1\right)&=\eta(u-1), \\
 \eta\left(\frac{b}{(u+1)a^2}+1\right)&=\eta(u+1).
 \end{split}
 \end{equation}
 Similarly, in Case $C_{1,-1}$, we must have 
 $\eta\left(2b+a^2 \left(u^2-1\right)\right)=\eta\left(-b-(u-1)a^2\right)= \eta\left(-b+(u+1)a^2\right)=1$, or equivalently,
 \begin{equation}
 \label{eq:case4}
 \begin{split}
 \eta\left(\frac{2b}{a^2 \left(u^2-1\right)}+1\right)&= \eta(u^2-1),\\
 \eta\left(\frac{-b}{(u-1)a^2}-1\right)&=\eta(u-1), \\
 \eta\left(\frac{-b}{(u+1)a^2}+1\right)&=\eta(u+1).
 \end{split}
 \end{equation}

We first take $p\equiv 5\pmod 8$ and $n$ odd (similarly, for the other cases). By Gauss' Reciprocity Law, we know that in these   fields, $2$ is a non-square (recall that 2 is a square in the field $\F_{p^n}$, $p$ odd
  if and only if either $p\equiv \pm 1\pmod 8$
or $n$ is even) and $-1$ is a square (since $p^n\equiv 1\pmod 4$ under our conditions). We shall be using that in the first part of our proof.

When $u=-3$, Case~$C_{1,1}$ reduces to
\[
\eta(b+2a^2)=\eta(b-2a^2)=\eta(a),
\]
Case~$C_{-1,-1}$ reduces to
\[
\eta(b+4a^2)=\eta(b-4a^2)=\eta(a),
\]
Case~$C_{-1,1}$ implies
\[
\eta(b-2a^2)=1\text{ and } [\eta(b+4a^2)=1\text{ or } \eta(b-2a^2)=1]
\]
(it is inclusive or, since we might have one or two solutions satisfying the conditions, here and in the next case)
and finally,
 Case~$C_{1,-1}$ implies
\[
\eta(b-4a^2)=-1\text{ and } [\eta(b-4a^2)=1\text{ or } \eta(b+2a^2)=1].
\]
We note that each case might contribute at most one solution. Summarizing,
\begin{align*}
C_{1,1}&: \eta(b+2a^2)=\eta(b-2a^2)=\eta(a)\\
C_{-1,-1}&: \eta(b+4a^2)=\eta(b-4a^2)=-\eta(a)\\
C_{-1,1}&: \eta(b-2a^2)=-\eta(b+4a^2)=1\\
C_{1,-1}&: \eta(b+2a^2)=-\eta(b-4a^2)=1.
\end{align*}
Thus, the number of solutions is  at most four, which is attained, as one can quickly check for some primes (for example, $p=461, n=1$). Later, we shall show that, in fact, for $q>125$, only the values $3,4$ are obtained.
 \end{proof}

We now continue with the following observation. 
Combining the cases $C_{1,1}$ and $C_{1,-1}$, we are seeking to show that 
$$(x+a)^{\frac{p^n+3}{2}}+u(x+a)^2-x^{\frac{p^n+3}{2}}-ux^2=b$$
has at least three solutions and thus the function $F$ is not APN.

Consider a fixed $u \in \mathbb{F}_q\setminus \{0,\pm 1 \}$.

The case $C_{1,1}$ provides a solution if and only if there exist $a,b,X,Y\in \mathbb{F}_q$, $a\neq 0$, such that 
$$\begin{cases}
    \frac{b}{2a(u+1)}-\frac{a}{2}=X^2\\
    \frac{b}{2a(u+1)}+\frac{a}{2}=Y^2.\\
\end{cases}
$$
On the other hand, the case $C_{-1,1}$ provides two solutions when there exist $a,b,Z,U,V,W,T\in \mathbb{F}_q$, $aZ\neq 0$, such that 
$$\begin{cases}
    a^2(u^2-1)-2b=Z^2\\
    a(u-1)+Z=2U^2\\
    a(u-1)-Z=2V^2\\
    a(u+1)+Z=2\xi W^2\\
    a(u+1)-Z=2\xi T^2.\\
\end{cases}$$
Putting altogether, we can observe that $F$ is not APN if, for a fixed $u \in \mathbb{F}_q\setminus \{0,\pm 1 \}$, there exist $a,b,X,Y,Z,U,V,W,T\in \mathbb{F}_q$, $aZ\neq 0$, satisfying the system 
\begin{equation*}
   \begin{cases}
    \frac{b}{2a(u+1)}-\frac{a}{2}=X^2\\
    \frac{b}{2a(u+1)}+\frac{a}{2}=Y^2\\
    a^2(u^2-1)-2b=Z^2\\
    a(u-1)+Z=2U^2\\
    a(u-1)-Z=2V^2\\
    a(u+1)+Z=2\xi W^2\\
    a(u+1)-Z=2\xi T^2\\
\end{cases} 
\end{equation*}
and such that the three roots
$$\frac{b-(u+1)a^2}{2a(u+1)}, \quad \frac{a (u-1)\pm\sqrt{a^2 \left(u^2-1\right)-2b}}{2}$$
are all distinct. This is implied by $b+a^2(u+1)\ne 0$.

Note that the above system is equivalent to 
\begin{equation}\label{System1}
\begin{cases}
    b=a(u+1)(2X^2+a)\\
    Y^2=a+X^2\\
    Z^2=a(u+1)((u-3)a-4X^2)\\
    U^2=\dfrac{u-1}{2}a+\frac{Z}{2}\\
    V^2=\dfrac{u-1}{2}a-\frac{Z}{2}\\
    W^2=\dfrac{a(u+1)}{2\xi}+\frac{Z}{2\xi}\\
    T^2=\dfrac{a(u+1)}{2\xi}-\frac{Z}{2\xi}.
\end{cases} 
\end{equation}

Our aim is to prove the existence of suitable solutions of the above system. To this end we will use an approach based on function fields over finite fields. 

Let $a$ be such that $a(u+1)$ is a square in $\mathbb{F}_q$. The solutions of the following system 
    \begin{equation}\label{System2_0}
    \begin{cases}
    b=a(u+1)(2X^2+a)\\
    Z^2=a(u+1)((u-3)a-4X^2)\\
    U^2=\dfrac{u-1}{2}a+\frac{Z}{2}\\
    V^2=\dfrac{u-1}{2}a-\frac{Z}{2}\\
    W^2=\dfrac{a(u+1)}{2\xi}+\frac{Z}{2\xi}\\
    T^2=\dfrac{a(u+1)}{2\xi}-\frac{Z}{2\xi}\\
    Y=\dfrac{\xi}{\sqrt{a(u+1)}}TW
\end{cases} 
\end{equation}
are also solutions of System~\eqref{System1}. 

We are now ready to put these together as a first step in the completion of our disproof of the conjecture.

\begin{thm}\label{Th:NoAPN}
    Let $q$ be an odd prime power, $u\in \mathbb{F}_q\setminus \{0,\pm1,3\}$, $a\in \mathbb{F}_q^*$ such that $a(u+1)$ is a square in $\mathbb{F}_q$, $\xi$ a fixed nonsquare in $\mathbb{F}_q$, that is, $\eta(a(u+1))=1,\eta(\xi)=-1$. The function field $\mathbb{K}(X,Y,Z,W,T)$ defined by 
    \begin{equation}
    \label{System2}
    \begin{cases}
    Z^2=a(u+1)((u-3)a-4X^2)\\
    W^2=\dfrac{a(u+1)}{2\xi}+\frac{Z}{2\xi}\\
    T^2=\dfrac{a(u+1)}{2\xi}-\frac{Z}{2\xi}\\
    Y=\dfrac{\xi}{\sqrt{a(u+1)}}TW\\
    U^2=\dfrac{u-1}{2}a+\frac{Z}{2}\\
    V^2=\dfrac{u-1}{2}a-\frac{Z}{2}
\end{cases} 
\end{equation}
has $\mathbb{F}_q$ as a field of constants.
\end{thm}
\begin{proof}
We rewrite the system above as
\begin{equation*}
    \begin{cases}
    Z=2\xi W^2-a(u+1)\\
    X^2= -\dfrac{\xi^2}{a(u+1)}W^4+\xi W^2-a\\
    T^2=\dfrac{a(u+1)}{\xi}-W^2\\
    Y=\dfrac{\xi}{\sqrt{a(u+1)}}TW\\
    U^2=\xi W^2-a\\
    V^2=-\xi W^2+au.
\end{cases}
\end{equation*}
Consider $\mathbb{K}_{0}:= \mathbb{K}(W)$. Clearly, $Z \in \mathbb{K}(W)$. We consider now $\mathbb{K}_{1}:= \mathbb{K}(X,W) $, where $X^2= -\frac{\xi^2}{a(u+1)}W^4+\xi W^2-a$. It is readily seen that $-\frac{\xi^2}{a(u+1)}W^4+\xi W^2-a$ is not a square in $\mathbb{K}_0$ and thus $\mathbb{K}_{1}$  is a Kummer extension of $\mathbb{K}_{0}$ with field of constants $\mathbb{F}_q$ by Theorem~\ref{Th_Kummer} and Lemma~\ref{ConstantField}. Let $\mathbb{K}_{2}:= \mathbb{K}(X,T,W) $, where $T^2=\frac{a(u+1)}{\xi}-W^2$. Since $\pm \sqrt{\frac{a(u+1)}{\xi}}$ are simple zeros of $\frac{a(u+1)}{\xi}-W^2$, which are not zeros of $-\frac{\xi^2}{a(u+1)}W^4+\xi W^2-a$, we conclude that above the places $P_{\pm \sqrt{\frac{a(u+1)}{\xi}}}\in \mathbb{K}_0$ there are exactly $4$ places in $\mathbb{K}_1$ which are simple zeros of $-\frac{\xi^2}{a(u+1)}W^4+\xi W^2$ and thus this function cannot be a square in $\mathbb{K}_1$. Again by  Theorem~\ref{Th_Kummer} and Lemma~\ref{ConstantField} we conclude that $\mathbb{K}_2$ is a Kummer extension of $\mathbb{K}_1$ and its field of constants $\mathbb{F}_q$. 

Consider now $\mathbb{K}_3 := \mathbb{K}(X,T,U,W)$, where $U^2=\xi W^2-a$. The zeros of $\xi W^2-a$ in $\mathbb{K}_2$ are not zeros of $X$ nor of $T$ and thus they lie over unramified places in the extension $\mathbb{K}_2:\mathbb{K}_0$. This shows that they are simple zeros for $\xi W^2-a$ and thus $\xi W^2-a$ is not a square in $\mathbb{K}_2$. By  Theorem~\ref{Th_Kummer} and Lemma~\ref{ConstantField} we conclude that $\mathbb{K}_3$ is a Kummer extension of $\mathbb{K}_2$ and its field of constants $\mathbb{F}_q$. 

Let $\mathbb{K}_4 := \mathbb{K}(X,T,U,V,W)$, where $V^2=-\xi W^2+au$. The zeros of $-\xi W^2+au$ in $\mathbb{K}_3$ are not zeros of $X$, nor of $T$, nor of $U$ and thus they lie over unramified places in the extension $\mathbb{K}_3:\mathbb{K}_0$. Arguing as before, we conclude that $\mathbb{K}_4$ is a Kummer extension of $\mathbb{K}_3$ and its field of constants $\mathbb{F}_q$.

To conclude the proof it is sufficient to note that $\mathbb{K}_5:=\mathbb{K}(X,Y,T,U,V,W) $ coincides with~$\mathbb{K}_4$.
\end{proof}

We now show that the conjecture of~\cite{BP24} is false, and not only there is no infinite family of APN functions, but in fact there are no APN functions besides those listed in~\cite[Table 5]{BP24}.
\begin{thm}
    Let $q$ be an odd prime power, $q>125$, and select $u \in \mathbb{F}_q\setminus \{0,\pm 1\}$. The polynomial $F(x)=x^{(q+3)/2}+ux^2$ is not APN.
\end{thm}
\begin{proof}
We first let $u=3$.  Although, one can also infer it from~Equation~\eqref{Table:cases2}, if $q\equiv 1\pmod 4$, our treatment of System~\ref{System2} requires $a(u+1)$ and $-4a(u+1)$ to be squares concurrently, which can happen for $q\equiv 1\pmod 4$, since $\eta(-1)=1$.
We removed $u=3$ from the statement of Theorem~\ref{Th:NoAPN} (and even below), since we wanted to treat the system globally, but the arguments also hold for $u=3,q\equiv 1\pmod 4$. 

Thus, we next let $u=3,q\equiv 3\pmod 4$, and so, $\eta(-1)=-1$. Further, if $n$ is even, then regardless of $p$, $2,3$ are quadratic residues and we get at least three viable solutions by taking $\eta(a)=1$, $b=-4a^2$, as we see from  Equation~\eqref{Table:cases2}. If $n$ is odd, $p\equiv 3,11\pmod {24}$, again, $2,3$ are quadratic residues and the same argument applies. If $n$ is odd, $p\equiv 23\pmod {24}$ (we removed $p\equiv 9\pmod {24}$, since we are in the case of $q\equiv 3\pmod 4$; similarly, we will also remove, from the next discussion, the cases $p\equiv 5,21\pmod {24}$), then $\eta(2)=1,\eta(3)=-1$. Taking $b=-4a^2$ with $\eta(a)=1$, at least three solutions of Equation~\eqref{Table:cases2} survive. If $n$ is odd, and $p\equiv 11 \pmod {24}$, then $\eta(2)=-1,\eta(3)=1$, taking again, $b=-4a^2$, with $\eta(a)=-1$, exactly three solutions survive. If $n$ is odd, and $p\equiv 19 \pmod {24}$, then $\eta(2)=\eta(3)=-1$, and $b=4a^2$, at least three solutions survive. Therefore, even when $u=3$, the function is not APN, for $p$ larger than~$29$ (see~\cite[Table 5]{BP24}, for small cases).

We now let $u\neq \pm 1,3$. Select $a\in \mathbb{F}_q^*$ such that $a(u+1)$ is a square in $\mathbb{F}_q$ and consider a fixed nonsquare $\xi\in \mathbb{F}_q$. By Theorem \ref{Th:NoAPN}, System \eqref{System2} defines a function field whose field of constants is $\mathbb{F}_q$. By direct checking, following the same notation as in the proof of Theorem \ref{Th:NoAPN}, by Theorem \ref{Th_Kummer} the genus of $\mathbb{K}_i$, $i=0,\ldots,5$, is $0$, $1$, $3$, $9$, $25$, $25$ respectively. There are at most $2^4$ places lying over $P_{\infty}$. Since we need the three roots to be distinct, $b+a^2(u+1)\neq 0$, together with $Z\neq 0$. Recalling that $b=a(u+1)(2X^2+a)$, the first above condition is equivalent to $X^2\neq -a$ and thus $-\frac{\xi^2}{a(u+1)}W^4+\xi W^2=0$.  There are at most $2^6$ places in $\mathbb{K}_5$ satisfying this constraint. Finally, $Z=0$ corresponds to $W^2=\frac{a(u+1)}{2\xi}$ and again there are at most $2^5$ places in $\mathbb{K}_5$ satisfying this constraint. Thus, the polynomial $F$ is not APN whenever the lower bound given by the Hasse-Weil bound exceeds $2^4+2^5+2^6$, that is 
$$q-50\sqrt{q}-111>0 \iff \sqrt{q}\geq 52.13, \text{ so }, q\geq 2719.$$

For the cases   when $125< q=p^n<2719$ (that is, outside~\cite[Table 5]{BP24}, which lists $q=5^3$ as the highest cardinality when the function is APN, for some specific values of $u$), we used   Magma~\textup{\cite{BCP97}} and found no other cases when the function is APN.
The claim follows.
\end{proof}

 \begin{rem}
Via Magma~\textup{\cite{BCP97}}, we found that, in addition to~\textup{\cite[Table 5]{BP24}}, there are other interesting examples of best differential uniformity functions. For example,  if $p=3$, $n=1$, $u=-1$, the function is PN; if $p=3$, $n=2$, $u=g,g^3,g^5,g^7$, the function is PN.  
 \end{rem}
 
In the appendix we display more computational data to display the differential spectrum for various dimensions and parameters~$u$.

\section{A related class of potential APNs}

Let $F(x)=x^{\frac{p^n-1}{2}+3}+u x^3$ on $\F_{p^n}$ with $p>3$, $q=p^n$. Computationally, we observed that, for some $u$ values, $F$ is APN for $p=5, 7, 11,13,19,23$ and $n=1$, as well as $p=5, n=2$, and has mostly low differential uniformity, for other small dimensions.

One surely wonders if there are infinitely many pairs $(p,n)$ for which $F$ is APN.
We shall show that that is not the case (at least when $q\equiv 1\pmod 3$).
The differential equation at $a\in\F_q^*,b\in\F_q$ is equivalent to 
\[
(x+a)^3(u+t_a)-x^3(u+t_x)=b.
\]
If $t_a=t_x=\epsilon\in\{\pm 1\}$, then, the above equation transforms into
\[
x^2+ax+\frac{a^2}{3}-\frac{b}{u+\epsilon}=0,
\]
of discriminant $\Delta=-\frac{a^2}{3}+\frac{4b}{3(u+\epsilon)}$. It is known that in $\F_{p^n}$   half of all nonzero elements of squares and half are non-squares. Regardless, since $\Delta$ is linear in $b$, then we conclude that one can always find $b$ to force $\Delta$ to be a nonzero square and hence our differential equation has two solutions, for some $a,b$.

If $t_a=-t_x=\epsilon\in\{\pm 1\}$, the differential equation is equivalent to 
\[
G(x):=x^3+\frac{3a(u+\epsilon)}{2\epsilon} x^2+\frac{3a^2(u+\epsilon)}{2\epsilon} x+\frac{a^3(u+\epsilon)-b}{2\epsilon}=0.
\]


In what follows we want to prove that if $q\equiv 1\pmod 3$ there exist $(a,b)\in \mathbb{F}_q^*\times \mathbb{F}_q$ such that $G(x)$ has three distinct roots in $\mathbb{F}_q$, all of them satisfying $t_x=1=-t_a$.

To this end, recall that when  $q\equiv 1\pmod 3$, the roots of $G(x)$ can be expressed in terms of the roots (in $\mathbb{F}_q$) of the Hessian polynomial $H(T)$ associated with $G(x)$. Now, 
$$H(T):=\frac{-9}{4}\left((-a^2u^2 + a^2)T^2+(-a^3u^2 + a^3 - 2b)T+(-abu - ab)\right)$$
whose discriminant is
$$\delta(a,b):= \frac{81}{16}\left(a^6u^4 - 2a^6u^2 + a^6 - 4a^3bu^3 + 4a^3bu + 4b^2\right).$$

According to \cite[Theorem 1.34]{Hirsc} the above cubic equation $G(x)=0$ possesses three (distinct) solutions in $\mathbb{F}_q$ if and only if $\delta(a,b)\in \square_{q}$ (the set of all squares in $\F_q$)  and $G(\beta_1)/G(\beta_2)$ is a cube in $\mathbb{F}_q$, where $\beta_1,\beta_2$ are the roots of $H(T)$, and the discriminant of $G$, say, $\Delta\neq 0$. 

We will make use of a direct expression for the three roots. Let $\eta$ be a fixed third root of unity in $\mathbb{F}_q$, and $e\in \mathbb{F}_q$ fixed with $e^3 =G(\beta_1)/G(\beta_2)$. The three roots of $G(x)$ are 
$$x_1:= \frac{\beta_2 e-\beta_1}{e-1}, \quad x_2:= \frac{\beta_2 \eta e-\beta_1}{\eta e-1}, \quad x_3:= \frac{\beta_2 \eta^2e-\beta_1}{\eta^2 e-1}.$$

Let $X^2=\delta(a,b)$, $\beta_1=\frac{a^3u^2-a^3+2b+X}{2a^2(1-u^2)}$, $\beta_2=\frac{a^3u^2-a^3+2b-X}{2a^2(1-u^2)}$. Then 
$$\frac{G(\beta_1)}{G(\beta_2)}=\frac{a^3 u^3 - a^3 u - 2b - 4X/9 }{a^3 u^3 - a^3 u - 2 b + 4X/9}.$$

Let $\xi \in \mathbb{F}_q$ be a fixed nonsquare. The function $F(x)=x^{\frac{p^n-1}{2}+3}+u x^3$ possesses at least three solutions if the following system 
$$
\begin{cases}
a^6u^4 - 2a^6u^2 + a^6 - 4a^3bu^3 + 4a^3bu + 4b^2 = 16X^2/81\\
\dfrac{a^3 u^3 - a^3 u - 2b - 4X/9 }{a^3 u^3 - a^3 u - 2 b + 4X/9}=Y^3\\
\dfrac{\beta_2 \eta^i e-\beta_1}{\eta^i e-1}=\xi Z_i^2, \ i=0,1,2,\\
\dfrac{\beta_2 \eta^i e-\beta_1}{\eta^i e-1}+a=W_i^2, \ i=0,1,2.
\end{cases}$$
has nontrivial solutions $(X,Y,Z_1,W_1,Z_2,W_2,Z_3,W_3,a,b)$ with $aXY \neq 0$, $Y\neq 1$.

Note that $\frac{\beta_2 \eta^i e-\beta_1}{\eta^i e-1}+a=W_i^2$ can be rewritten as $\xi Z_i^2+a=W_i^2$, $i=0,1,2$. 

From the second equation one gets 
$$b=\frac{9 a^3 u^3 Y^3 - 9 a^3 u^3 - 9 a^3 u Y^3 + 9 a^3 u + 4 Y^3 X + 4 X}{18(Y^3-1)}.$$

Combining with the other equations one obtains
$$\begin{cases}
-81(Y^3-1)^2(u^2-1)^3a^6 + 64Y^3X^2=0\\
(Y-1) \left( 9(u^2-1)a^2(a u + a + 2\xi Z_0^2)Y^3-8XY^2-8X Y-9(u^2-1)a^2(a u + a + 2 \xi Z_0^2)\right)=0\\
\xi Z_0^2+a=W_0^2\\
 (Y-\eta^2)\left(9(u^2-1)a^2(a u + a + 2 Z_1^2)Y^3 -8\eta^2 XY^2-8 \eta X Y-9(u^2-1)a^2(a u + a + 2 Z_1^2)\right)=0\\
\xi Z_1^2+a=W_1^2\\
   (Y-\eta)\left(9(u^2-1)a^2(a u + a + 2 \xi Z_2^2)Y^3-8 \eta XY^2-8\eta^2 X Y-9(u^2-1)a^2(a u + a + 2 \xi Z_2^2)\right)=0\\
   \xi Z_2^2+a=W_2^2.
\end{cases}$$

Let us discard the factors $(Y-1) $, $(Y-\eta)$, $(Y-\eta^2)$. From the second equation we obtain
$$X = \frac{9(a u + a + 2\xi Z_0^2)a^2(u^2-1)(Y^3-1)}{8Y(Y+1)}$$
and thus the system above, using  also $\xi Z_0^2 =-a +W_0^2$,   reads (after discarding factors such as $a$, $Y$, $Y+1$, and $(Y^3-1)$) 
$$\begin{cases}
\xi Z_0^2 =-a +W_0^2\\
(uY^2 + uY + u + Y^2 + 3Y + 1)(u-1)a^2  -4YW_0^2(u-1)a- 4W_0^4Y=0\\
3auY - 3\eta a u + (-2\eta - 1) a Y + (\eta + 2) a + (2\eta + 4)W_0^2 Y +
    (-4\eta - 2)W_0^2 + (-2\eta + 2)(Y+1)Z_1^2=0\\
    W_1^2 =a +\xi Z_1^2 \\
3auY + (3\eta + 3)au + (2\eta + 1)aY + (-\eta + 1)a + (-2\eta + 
    2)W_0^2Y + (4\eta + 2) W_0^2 + (2\eta + 4)(Y+1)\xi Z_2^2 Y=0\\
W_2^2=\xi Z_2^2+a.
\end{cases}$$

Finally we can combine the fourth and the third,  the sixth and fifth equation to get the following 
$$\begin{cases}
(uY^2 + uY + u + Y^2 + 3Y + 1)(u-1)a^2  -4YW_0^2(u-1)a- 4W_0^4Y=0\\
Z_0^2 =\displaystyle\frac{-a +W_0^2}{\xi}\\
Z_1^2=-\displaystyle\frac{(2 \eta + 4) a u Y + (-2 \eta + 2) a u - 2 \eta a Y + (2 \eta + 2) a + 
        (4 \eta + 4) W_0^2 Y - 4 \eta W_0^2}{4(Y+1)}\\
W_1^2 =-\displaystyle\frac{(2\eta + 4) \xi a u Y  + (-2\eta + 2)\xi a u  - 2\eta \xi a Y - 4 a Y + 
        (2\eta + 2) a xi - 4 a + (4\eta + 4)\xi  W_0^2 Y - 4\eta\xi  W_0^2 }{4(Y+1)}\\
Z_2^2=\displaystyle\frac{(2 \eta - 2) a u Y + (-2 \eta - 4) a u + (-2 \eta - 2) a Y + 2 \eta a + 
    4 \eta W_0^2 Y + (-4 \eta - 4) W_0^2}{4\xi(Y+1)}\\
W_2^2=\displaystyle\frac{(2\eta - 2)\xi a u Y  + (-2\eta - 4)\xi a u  + (-2\eta + 2) \xi a Y  + (2\eta + 
    4) a \xi + 4\eta \xi W_0^2 Y  + (-4\eta - 4)\xi W_0^2 }{4\xi(Y+1)}.
\end{cases}$$

Let us put all this discussion together.
\begin{thm}
Let $q\equiv 1 \pmod {3}$. Choose a nonsquare $\xi\in \mathbb{F}_q$ and consider $u\notin \{\pm1,\pm \sqrt{3},\pm 2\}$. If $q$ is large enough then $F(x)=x^{\frac{p^n-1}{2}+3}+u x^3$ is not APN.
\end{thm}
\begin{proof}
As a notation, recall that $\mathbb{K}$ is the algebraic closure of $\mathbb{F}_q$. Consider first the case  $u$ not being a root of 
\begin{eqnarray*}
    h(u)&:=& (u^2 - 4)\Big((\xi^2-\xi)u + \xi^2 - \xi + 2/3\Big)\Big(\xi^2 u^2 +(-\xi^2+\xi)u - 2\xi^2 + \xi - 2\Big) \\
    &&\cdot (u-1)(u^2 - 3)(\xi u + \xi - 1)\Big(\xi u - 2/3\xi^2 + 1/3\xi - 2/3\Big)\\
    &&\cdot \Big(\xi^2 u^2 + \xi u - 3\xi^2 + 3\xi - 2\Big)\Big(\xi^2 u^2 +(2\xi -\xi^2)u - 2\xi^2 + 2\xi - 2\Big).
\end{eqnarray*} 
By the discussion above, if the system 
$$\begin{cases}
(uY^2 + uY + u + Y^2 + 3Y + 1)(u-1)a^2  -4YW_0^2(u-1)a- 4W_0^4Y=0\\
Z_0^2 =\displaystyle\frac{-a +W_0^2}{\xi}\\
Z_1^2=-\displaystyle\frac{(2 \eta + 4) a u Y + (-2 \eta + 2) a u - 2 \eta a Y + (2 \eta + 2) a + 
        (4 \eta + 4) W_0^2 Y - 4 \eta W_0^2}{4(Y+1)}\\
W_1^2 =-\displaystyle\frac{(2\eta + 4) \xi a u Y  + (-2\eta + 2)\xi a u  - 2\eta \xi a Y - 4 a Y + 
        (2\eta + 2) a xi - 4 a + (4\eta + 4)\xi  W_0^2 Y - 4\eta\xi  W_0^2 }{4(Y+1)}\\
Z_2^2=\displaystyle\frac{(2 \eta - 2) a u Y + (-2 \eta - 4) a u + (-2 \eta - 2) a Y + 2 \eta a + 
    4 \eta W_0^2 Y + (-4 \eta - 4) W_0^2}{4\xi(Y+1)}\\
W_2^2=\displaystyle\frac{(2\eta - 2)\xi a u Y  + (-2\eta - 4)\xi a u  + (-2\eta + 2) \xi a Y  + (2\eta + 
    4) a \xi + 4\eta \xi W_0^2 Y  + (-4\eta - 4)\xi W_0^2 }{4\xi(Y+1)}.
\end{cases}$$
possesses a solution $(a,Y,Z_0,Z_1,Z_2,W_0,W_1,W_2)\in \mathbb{F}_q^8$ such that $a\neq 0$, $(Y^3-1)Y(Y+1)\neq 0$, $(u+1)Y^2 +(3- 2u)Y + u   + 1 \neq 0$, then $F(x)$ is not APN. Note that $(u+1)Y^2 +(3- 2u)Y + u   + 1 \neq 0$ and $Y(Y+1)\neq 0$ yield $X\in \mathbb{F}_q^*$. To this end, fix an element $W_0\neq 0$ in $\mathbb{F}_q$. 

We will show that such a system defines a function field whose field of fractions is $\mathbb{F}_q$. This will provide, asymptotically, a negative answer for the APNness of the function $F(x)$.

Consider first the equation, 
$$(uY^2 + uY + u + Y^2 + 3Y + 1)(u-1)a^2  -4YW_0^2(u-1)a- 4W_0^4Y=0.$$
The discriminant with respect of $a$ is 
$$64 Y(Y+1)^2 W_0^4(u^2-1) $$
and it is, clearly, a nonsquare in $\mathbb{K}_0:= \mathbb{K}(Y)$. Thus, by Theorem~\ref{Th_Kummer} and Lemma~\ref{ConstantField} the field of constants of $\mathbb{K}_1=\mathbb{K}_0(a)$ is $\mathbb{F}_q$. We want now to prove that the other $5$ equations define always Kummer extensions with field of constants $\mathbb{F}_q$.

To this end, consider the subsequent equations written as
$$Z_0^2=\phi_1(a,Y),\ldots, W_2^2=\phi_5(a,Y)$$
and denote by $\mathbb{K}_{2}\subset \cdots \subset \mathbb{K}_6$ the corresponding function fields. 

To our aim, it is sufficient to show that for each $i=1,\ldots,5$ there exists at least a place in $\mathbb{K}_1$ which is a simple zero for $\phi_i$ and a nonzero for each $\phi_j$, $j\neq i$. This will ensure the existence of a place in $\mathbb{K}_{i+1}$ which is a simple zero for $\phi_i$. 

Thus we will check the resultant between the numerators of the functions $\phi_i(a,Y)$ and $(uY^2 + uY + u + Y^2 + 3Y + 1)(u-1)a^2  -4YW_0^2(u-1)a- 4W_0^4Y=0$ with respect to $a$. We want to prove that each of them has a nonrepeated linear factor in $Y$ (different from $Y$ and $(Y+1)$), which is not a factor of any other resultant. 

We list below the factorizations of these resultants 
\begin{eqnarray*}
    &&(u+1)W_0^4\Big((u-1) Y^2 + (u-3) Y + u   - 1\Big),\\
    &&(Y+1)^4W_0^4(u+1)\Big((u-1) Y^2 + (-\eta - 1)uY + \eta u  + (3\eta + 3) Y - \eta \Big),\\
    &&(Y + 1)^4W_0^4\Big(\xi^2(u^2 - 1) Y^2-(\eta+1)(\xi ^2 u^2 - 2\xi^2 u + 4\xi u - 3\xi^2 + 4\xi - 4)Y+\eta\xi^2(u^2 - 1) \Big),\\
     &&\xi^2 (Y + 1)^4W_0^4(u + 1)\Big((u-1)Y^2 + (\eta u-3) Y + (-\eta - 1) u   + \eta + 1\Big),\\
     &&\xi^4 (Y + 1)^4W_0^4(u - 1)\Big((u+1) Y^2 + (\eta u+3) Y + (-\eta - 1) u   - \eta - 1\Big).
\end{eqnarray*}
As it can be easily checked, there is common zero among them (apart from $-1$) only if $u$ is a root of $h(u)$. If $u$ is not a root of $h(u)$, none of the degree-2 factors above has $Y=-1$ as a root. 

Thus, there exists a simple zero of each $\phi_i$ which is not a zero (nor a pole) of $\phi_j$, $j\neq i$, for each $i=1,\ldots,5$, and the claim follows. 

Suppose that $u$ is a  zero of $h(u)$ distinct from $\pm1,\pm \sqrt{3},\pm 2$. We can choose a different nonsquare $\xi\in \mathbb{F}_q$ to obtain the desired result.

 
\end{proof}

\begin{rem}
Presumably, the same outcome will happen for $q\equiv 2\pmod 3$. Recall (see~\textup{\cite[Theorem 1, Corollary 2.9]{MS87}}, or \textup{\textup~\cite{Hirsc}}) that if  $\F_q$ is a field of  characteristic different from $3$, then $f(x)= ax^3+bx^2+cx+d$,$a\neq 0$,   permutes $\F_q$ if and only if $b^2=3ac$, and $q\equiv 2\pmod 3$. Thus, the case of $q\equiv 2\pmod 3$ will be slightly more involved (though, doable), since one needs several cases necessary to handle the roots of a cubic under this modularity condition on~$q$.
\end{rem}

\section{The Difference Distribution Table of a general class of functions in terms of Weil sums}

Since, in general, finding,  theoretically or computationally,  the differential spectrum of a  function is quite difficult, various methods have been previously proposed, one of which based on characters  seems to work very well (proposed and argued in~\cite{PS23} that it achieves speeds of more than ten times as much the classical method).
\begin{thm}
Let $F(x)=x^{\frac{p^n-1}{2}+p^k+1}+u x^{p^j+1}$, $0\leq j\leq k<n$, on $\F_{p^n}$. The Difference Distribution Table entries at $(a,b)\in\F_{p^n}^*\times \F_{p^n}$ are given by
\begin{align*}
\cN_{a,b}
&= p^{-n}\sum_{\alpha,x\in\F_{p^n}} \chi_1\left(\alpha (u+\epsilon) x^{p^k+1}-\alpha(u+\epsilon\mu) x^{p^j+1} \right.\\
&\qquad\qquad \left.+ \left(\left(\alpha a(u+\epsilon)\right)^{p^{n-k}}   + \alpha a^{p^k} (u+\epsilon) \right) x +\alpha\left((u+\epsilon) a^{p^k+1}-b\right)\right).
\end{align*}
\end{thm}
\begin{proof}
As before, we let $\eta(x)=t_x$ and $\eta(x+a)=t_a$.
The differential
equation now becomes
\[
(x+a)^{p^k+1}(u+t_a)-x^{p^j+1} (u+t_x)=b,
\]
and with notations $t_a=\mu t_x=\epsilon$, where both $\mu,\epsilon\in\{\pm 1\}$, we obtain
\begin{equation}
\label{eq:DOjk}
(u+\epsilon) x^{p^k+1}-(u+\epsilon \mu) x^{p^j+1} +(u+\epsilon) (a x^{p^k} +a^{p^k} x) + (u+\epsilon) a^{p^k+1}-b=0.
\end{equation}

The number  $N(b)$ of solutions $(x_1,\ldots,x_n)\in\F_q^n$ ($b$ is fixed) of an equation $f(x_1,\ldots,x_n)=b$ is~\cite{LN97}
\begin{align*}
\cN(b)
&= \frac{1}{q}\sum_{x_1,\ldots,x_n\in \F_q}\sum_{\alpha\in\F_q} \chi_1\left(\alpha \left( f(x_1,\ldots,x_n)- b\right)\right).
\end{align*}
If $\cN_{a,b}$ is the number solutions of solutions for~\eqref{eq:DOjk}, then  
\begin{align*}
p^n\,\cN_{a,b}&=\sum_{x\in\F_{p^n}}\sum_{\alpha\in\F_{p^n}} \chi_1\left(\alpha\left( (u+\epsilon) x^{p^k+1}-(u+\epsilon\mu) x^{p^j+1} +(u+\epsilon) (a x^{p^k} +a^{p^k} x) + (u+\epsilon) a^{p^k+1}-b\right) \right)\\
=& \sum_{\alpha\in\F_{p^n}} \chi_1\left(\alpha\left((u+\epsilon) a^{p^k+1}-b\right)\right)\sum_{x\in\F_{p^n}}\chi_1\left(\alpha\left( (u+\epsilon) x^{p^k+1}-(u+\epsilon \mu) x^{p^j+1} +(u+\epsilon) (a x^{p^k} +a^{p^k} x) \right) \right).
\end{align*}
We now concentrate on the inner sum, for a fixed $\alpha$, and we write
\begin{align*}
& \sum_{x\in\F_{p^n}}\chi_1\left(\alpha\left( (u+\epsilon) x^{p^k+1}-(u+\epsilon \mu) x^{p^j+1} +(u+\epsilon) (a x^{p^k} +a^{p^k} x) \right) \right)\\
&=\sum_{x\in\F_{p^n}}\chi_1\left(\alpha\left( (u+\epsilon) x^{p^k+1}-(u+\epsilon \mu) x^{p^j+1}  \right) \right)\chi_1\left(\alpha(u+\epsilon) (a x^{p^k} +a^{p^k} x) \right)\\
&=\sum_{x\in\F_{p^n}}\chi_1\left(\alpha\left( (u+\epsilon) x^{p^k+1}-(u+\epsilon \mu) x^{p^j+1}  \right)\right)\chi_1\left(\alpha a(u+\epsilon)  x^{p^k} \right)  \chi_1\left(\alpha a^{p^k} (u+\epsilon)  x \right)\\
&=\sum_{x\in\F_{p^n}}\chi_1\left(\alpha\left( (u+\epsilon) x^{p^k+1}-(u+\epsilon \mu) x^{p^j+1}  \right)\right)\chi_1\left(\left(\alpha a(u+\epsilon)\right)^{p^{n-k}}  x \right)  \chi_1\left(\alpha a^{p^k} (u+\epsilon)  x \right)\\
&=\sum_{x\in\F_{p^n}}\chi_1\left(\alpha\left( (u+\epsilon) x^{p^k+1}-(u+\epsilon\mu) x^{p^j+1}  \right)\right)\chi_1\left(\left( \left(\alpha a(u+\epsilon)\right)^{p^{n-k}}   + \alpha a^{p^k} (u+\epsilon) \right) x \right)\\
&=\sum_{x\in\F_{p^n}}\chi_1\left(\alpha (u+\epsilon) x^{p^k+1}-\alpha(u+\epsilon\mu) x^{p^j+1}   + \left(\left(\alpha a(u+\epsilon)\right)^{p^{n-k}}   + \alpha a^{p^k} (u+\epsilon) \right) x \right).
\end{align*}
Putting these sums together, it shows the result.
\end{proof}

We make some further comments on the above sums.
Let $A_1:=\alpha (u+\epsilon)$, $A_2:=-\alpha (u+\epsilon\mu)$, $B:=\left(\alpha a(u+\epsilon)\right)^{p^{n-k}}   + \alpha a^{p^k} (u+\epsilon) $, and let
$S_\alpha:=\displaystyle\sum_{x\in\F_{p^n}}\chi_1\left(A_1x^{p^k+1}+A_2 x^{p^j+1}+Bx \right)$. We write
\allowdisplaybreaks
\begin{align*}
|S_\alpha|^2&=S_\alpha\cdot \bar S_\alpha
=\sum_{x\in\F_{p^n}}\chi_1\left(A_1x^{p^k+1}+A_2 x^{p^j+1}+Bx \right)\overline{\sum_{y\in\F_{p^n}}\chi_1\left(A_1 y^{p^k+1}+A_2 y^{p^j+1}+By \right)}\\
&= \sum_{x\in\F_{p^n}}\chi_1\left(A_1x^{p^k+1}+A_2 x^{p^j+1}+Bx \right) \sum_{y\in\F_{p^n}}\chi_1\left(-A_1 y^{p^k+1}-A_2 y^{p^j+1}-By \right)\\
&=\sum_{x,y\in\F_{p^n}}\chi_1\left(A_1\left(x^{p^k+1}-y^{p^k+1}\right)+A_2 \left(x^{p^j+1}-y^{p^j+1}\right)+B(x-y) \right)\\
&=\sum_{y,z\in\F_{p^n}}\chi_1\left(A_1\left((y+z)^{p^k+1}-y^{p^k+1}\right)+A_2 \left((y+z)^{p^j+1}-y^{p^j+1}\right)+Bz\right)\\
&=\sum_{y,z\in\F_{p^n}}\chi_1\left(A_1\left(z^{p^k+1}+z^{p^k}y+z y^{p^k}\right)+A_2 \left(z^{p^j+1}+z^{p^j}y+z y^{p^j}\right)+Bz\right)\\
&= \sum_{z\in\F_{p^n}}\chi_1\left(A_1 z^{p^k+1}+A_2  z^{p^j+1}+Bz\right)
\sum_{y\in\F_{p^n}} \chi_1\left(A_1\left( z^{p^k}y+z y^{p^k}\right)+A_2 \left( z^{p^j}y+z y^{p^j}\right) \right)\\
&= \sum_{z\in\F_{p^n}}\chi_1\left(A_1 z^{p^k+1}+A_2  z^{p^j+1}+Bz\right)
\sum_{y\in\F_{p^n}} \chi_1\left(\left(A_1\left( z^{p^k}+z^{p^{n-k}} \right)+A_2 \left( z^{p^j}+z^{p^{n-j}}\right) \right)y\right).
\end{align*}
Since the inner sum is of a linear function (in $y$), the sum is zero unless the coefficient of $y$ is zero, $E(z)=0$, where $E(z):= A_1\left( z^{p^k}+z^{p^{n-k}} \right)+A_2 \left( z^{p^j}+z^{p^{n-j}}\right) $. Thus,
\begin{align*}
|S_\alpha|^2=S_\alpha\cdot \bar S_\alpha
&=p^n\sum_{z\in\F_{p^n},E(z)=0}\chi_1\left(A_1 z^{p^k+1}+A_2  z^{p^j+1}+Bz\right).
\end{align*}

 When $k=j$, then $S_\alpha=\displaystyle\sum_{x\in\F_{p^n}}\chi_1\left(\alpha\epsilon(1-\mu)x^{p^k+1}+Bx \right)$ and $E(z)=(A_1+A_2)\left( z^{p^k}+z^{p^{n-k}} \right)$. If further $\mu=1$, or $\alpha=0$, then $E(z)$ is identically~$0$, and if $\mu=-1$ and $\alpha\neq 0$, then $E(z)=2\alpha\epsilon\left( z^{p^k}+z^{p^{n-k}} \right)$. Thus, if $\mu=1$, or $\alpha=0$, then 
 $S_\alpha\cdot \bar S_\alpha=p^n S_\alpha$, therefore, either $S_\alpha=0$, or $S_\alpha=p^n$. 
 
 If $\mu=-1$ and $\alpha\neq 0$, then $E(z)=0$ is equivalent to $z^{p^{2k}}+z=0$.
 It is known~\textup{\cite{ZWW20}}   that a linearized polynomial of the form $L(x)=x^{p^r}+\gamma x\in\F_{p^n}$ is a permutation polynomial if and only if the relative norm $N_{\F_{p^n}/\F_{p^d}}(\gamma)\neq 1$, that is, $(-1)^{n/d} \gamma^{(p^n-1)/(p^d-1)}\neq 1$, where $d=\gcd(n,r)$. In our case, $r=2k$ and  $\gamma=1$, and so, if $\frac{n}{gcd(n,2k)}\equiv 1\pmod 2$, then $z^{p^{2k}}+z$ is a permutation polynomial with $z=0$ as its only root. Thus, if $\mu=-1$, $\alpha\neq 0$, and $\frac{n}{gcd(n,2k)}$ is odd, then $S_\alpha=p^n$.

We thus showed the next corollary.
\begin{cor}
For the function in our prior theorem under $k=j$,  if $\mu=-1$, $\alpha\neq 0$, and $\frac{n}{gcd(n,2k)}$ is odd, then $\cN_{a,\alpha(u+\epsilon)a^{p^{k+1}}}=p^n$, and so, the function cannot be APN.
\end{cor}

\section{The case \texorpdfstring{$j=k$}{j=k}} 

When $j=k$ in the general class of the prior section, we can show a stronger result than the previous corollary.
\begin{thm}
We let $F(x)=x^{\frac{p^n-1}{2}+p^k+1}+u x^{p^k+1}$ on $\F_{p^n}$, where $k<n$, $u\ne \pm 1$, $d=\gcd(n,k)$, $q=p^k,Q=p^n$. Then $F(X)$ is not APN.
\end{thm}
\begin{proof}
For given $a\in\F_p^*, b\in\F_p$ we need to look at the differential equation $F(x+a)-F(x)=b$, that is,
\[
(x+a)^{\frac{p^n-1}{2}+p^k+1}+u(x+a)^{p^k+1}-x^{\frac{p^n-1}{2}+p^k+1}-ux^{p^k+1}=b.
\]
Denoting $t_a=\eta(x+a),t_x=\eta(x)$,   the equation above becomes
\begin{equation}
\label{eq:main2}
 (x+a)^{p^k+1} (t_a+u)-x^{p^k+1}(t_x+u)=b.
\end{equation}

We now distinguish four cases, which are displayed in the next table (we let $\epsilon=1,0$, if $-1$ is a $p^k-1$ power in $\F_{p^n}$, respectively, not a power).
{
\[
\begin{array}{|l||r|r|c|c|}
\hline \text { Case } & t_a & t_x & \text { Equation \eqref{eq:main2} } & \textrm{Number of roots}   \\
\hline 
D_{1,1} & 1 & 1 & x^{p^k}+a^{p^k-1}x+a^{p^k}-\frac{b}{u+1}=0 & \leq\epsilon\cdot (p^d-1)  \\
\hline 
D_{-1,-1} & -1 & -1 &  x^{p^k}+a^{p^k-1}x+a^{p^k}-\frac{b}{u-1}=0  & \leq\epsilon\cdot (p^d-1)  \\
\hline 
D_{1,-1} & 1 & -1 & x^{p^k+1}+\frac{a(1+u)}{2}x^{p^k}+\frac{a^{p^k}(1+u)}{2} x +\frac{a^{p^k+1}(1+u)-b}{2}=0 & N_1     \\
\hline 
D_{-1,1} & -1 & 1 &  x^{p^k+1}+\frac{a(1-u)}{2}x^{p^k}+\frac{a^{p^k}(1-u)}{2} x +\frac{a^{p^k+1}(1-u)+b}{2}=0    &  N_2   \\
\hline
\end{array}
\]
}

We shall look at the potential $N_1,N_2$ next, finding parameters $a,b$, for which either $N_1,N_2$ are greater than~$2$.
With $r=\frac{a(1+u)}{2},s=\frac{a^q(1+u)}{2}, t=\frac{a^{p^k+1}(1+u)-b}{2}$, for case $D_{1,-1}$, respectively, $r=\frac{a(1-u)}{2},s=\frac{a^q(1-u)}{2}, t=\frac{a^{p^k+1}(1-u)+b}{2}$, for case $D_{-1,1}$, 
and using the substitution 
$x=(s-r^q)^{\frac1q}X-r$, both equations in cases $D_{1,-1}$ and $D_{-1,1}$ become
\begin{equation}
\label{eq:kim1}
X^{q+1}+X+A=0,
\end{equation}
where 
\begin{align*} 
A&=\left(u^q-u\right)^{-\frac{q+1}{q}}\left(-2 b\, a^{-q-1} - u^2+1\right)=\frac{-2 b\, a^{-q-1} - u^2+1}{\left(u-u^{\frac{1}{q}}\right)^{q+1}}, \text{ for case }D_{1,-1},\\ A&=\left(u^q-u\right)^{-\frac{q+1}{q}}\left(2 b\, a^{-q-1} - u^2+1 \right)=\frac{2 b\, a^{-q-1} - u^2+1}{\left(u-u^{\frac{1}{q}}\right)^{q+1}}, \text{ for case } D_{-1,1}.
\end{align*}
Via \cite[Theorem 8]{Kim1}, we know that Equation~\eqref{eq:kim1} has $p^d+1$ roots if and only if there exists $U\in\F_Q\setminus \F_{p^{2d}}$ such that $A=\frac{(U-U^q)^{q^2+1}}{(U-U^{q^2})^{q+1}}$, in which case, those $p^d+1$ roots are given by
\begin{align*}
x_0=\frac{-1}{1+(U-U^q)^{q-1}},\ x_{\alpha}=\frac{-(U+\alpha)^{q^2-q}}{1+(U-U^q)^{q-1}}, \ \alpha\in\F_{p^d}.
\end{align*}
Regardless, of what the chosen $U\in\F_Q\setminus \F_{p^{2d}}$ is, since $A$ is linear in  $b$, then one is always able to find a value of $b$ such that $A=\frac{(U-U^q)^{q^2+1}}{(U-U^{q^2})^{q+1}}$.

We can force $x_0$, $x_{1}$, and $x_{-1}$ to be roots (asymptotically). Thus, we need
\begin{equation}\label{EquazioneGen}
\begin{cases}
 \frac{-1}{1+(U-U^q)^{q-1}}=\xi X^2\\
\frac{-1}{1+(U-U^q)^{q-1}}+a=Y^2\\ 
 \frac{-(U+1)^{q^2-q}}{1+(U-U^q)^{q-1}}=\xi Z^2\\
\frac{-(U+1)^{q^2-q}}{1+(U-U^q)^{q-1}}+a=V^2\\
 \frac{-(U-1)^{q^2-q}}{1+(U-U^q)^{q-1}}=\xi W^2\\
\frac{-(U-1)^{q^2-q}}{1+(U-U^q)^{q-1}}+a=T^2.
\end{cases}
\end{equation}


Combining the first and the third equation we get
$$\xi X^2(U+1)^{q^2-q}=\xi Z^2$$
 and thus $Z=\pm X (U+1)^{(q^2-q)/2}$.
 With the same argument, $W=X (U-1)^{(q^2-q)/2}$. Thus it is enough to show the existence of solutions of the following system 
$$
\begin{cases}
 \frac{-1}{1+(U-U^q)^{q-1}}=\xi X^2\\
\frac{-1}{1+(U-U^q)^{q-1}}+a=Y^2\\ 
\frac{-(U+1)^{q^2-q}}{1+(U-U^q)^{q-1}}+a=V^2\\
\frac{-(U-1)^{q^2-q}}{1+(U-U^q)^{q-1}}+a=T^2.\\
\end{cases}
$$ 

Clearly, all the roots of $1+(U-U^q)^{q-1}$ are distinct and so the poles of $\frac{-1}{1+(U-U^q)^{q-1}}$ are simple. This shows that 
$$\frac{-1}{1+(U-U^q)^{q-1}}=\xi X^2$$
is absolutely irreducible and $\mathbb{K}(X,U): \mathbb{K}(U) $, where $\mathbb{K}$ is the algebraic closure of $\mathbb{F}_q$, is a Kummer extension of the rational function field $\mathbb{K}(U) $ by Theorem~\ref{Th_Kummer}. By  Lemma~\ref{ConstantField} the field of constants of $\mathbb{K}(X,U)$ is $\mathbb{F}_q$. 
Consider the zeros of 
$$\phi_1:=\frac{-1}{1+(U-U^q)^{q-1}}+a, \psi_2:=\frac{-(U+1)^{q^2-q}}{1+(U-U^q)^{q-1}}+a, \phi_3:=\frac{-(U-1)^{q^2-q}}{1+(U-U^q)^{q-1}}+a.$$
They are roots of  
\begin{eqnarray*}\psi_1(U)&:=&-1+a(1+(U-U^q)^{q-1}),\\
\psi_2(U)&:=&-(U+1)^{q^2-q}+a(1+(U-U^q)^{q-1}),\\ \psi_3(U)&:=&-(U-1)^{q^2-q}+a(1+(U-U^q)^{q-1}),
\end{eqnarray*}
respectively. 

Since we can suppose that $a \neq 0,1$, all the roots of $\psi_1,\psi_2,\psi_3$ are distinct. In fact 
$$\psi^{\prime}_i(U)=-a(U-U^q)^{q-2}$$
and thus repeated roots can only belong to $\mathbb{F}_q$. On the other hand, $U\in \mathbb{F}_q$ being a root of $\psi_i$ yields either $a=0$ or $a=1$. 

Also, the zeros of $\psi_i$ and $\psi_j$, $i\neq j$, are  distinct. If $\psi_1$ and $\psi_2$ or $\psi_1$ and $\psi_3$ share a root, then such a root $z$  satisfies $(z+1)^{q-1}=1$ or $(z-1)^{q-1}=1$, and thus $z \in \mathbb{F}_q$. We already showed that this is not possible. If $\psi_2$ and $\psi_3$ share a root $z$, then $z=(1+\lambda)(\lambda-1)$, for some $\lambda$ in $\mathbb{F}_{q}\setminus \{0,\pm1\}$ and thus $z \in \mathbb{F}_q$, again a contradiction to $a\neq0,1$. 
Consider the function field extensions 
$$ \mathbb{K}(Y,X,U): \mathbb{K}(X,U), \mathbb{K}(V,Y,X,U): \mathbb{K}(Y,X,U), \mathbb{K}(T,V,Y,X,U): \mathbb{K}(V,Y,X,U),$$
defined by $Y^2=\phi_1$, $V^2=\phi_2$, and $T^2=\phi_3$ respectively.
From the argument above  each $\phi_i$ is not a square in the corresponding function field and thus by Theorem~\ref{Th_Kummer} and   Lemma~\ref{ConstantField}  each of the above extensions are Kummer extensions with field of constants $\mathbb{F}_q$.

This shows that, if $q$ is large enough, there are instances of $U,X,Y,Z,V,W,T$ satisfying System \eqref{EquazioneGen} and thus Equation \eqref{eq:main1} admits 3 solutions and $F(x)$ is not APN.
\end{proof}

\section*{Acknowledgements}
The second-named author (PS) would like to thank  the first-named author (DB) for the invitation at the Dipartimento di Matematica e Informatica  at Universit\`a degli Studi di Perugia, Italy, and the great working conditions while this paper was being written. The first-named author (DB) thanks the Italian National Group for Algebraic and Geometric Structures and their Applications (GNSAGA—INdAM) which supported the research.

\section*{Appendix}

We now display some computational data displaying the distribution of differential uniformity (DU) for various dimensions and parameters $u$. The notation $a^b$ means that the uniformity $a$ has frequency $b$.
\begin{table*}[ht]
    \centering
    \renewcommand{\arraystretch}{1.3}
    \begin{tabular}{c l}
        \hline
        $p$ &$u$\ :\ Differential Spectrum \\
        \hline
        5 & $2: \{ 0^{4}, 1^{12}, 2^{4} \}$ \\
          & $3: \{ 0^{4}, 1^{12}, 2^{4} \}$ \\
        \hline
        7 & $2: \{ 0^{12}, 1^{24}, 3^{6} \}$ \\
          & $3: \{ 0^{12}, 1^{18}, 2^{12} \}$ \\
          & $4: \{ 0^{12}, 1^{18}, 2^{12} \}$ \\
          & $5: \{ 0^{12}, 1^{24}, 3^{6} \}$ \\
        \hline 
        11 & $2: \{ 0^{40}, 1^{30}, 2^{40} \}$ \\
           & $3: \{ 0^{50}, 1^{20}, 2^{30}, 3^{10} \}$ \\
           & $4: \{ 0^{30}, 1^{60}, 2^{10}, 3^{10} \}$ \\
           & $5: \{ 0^{40}, 1^{50}, 3^{20} \}$ \\
           & $6: \{ 0^{40}, 1^{50}, 3^{20} \}$ \\
           & $7: \{ 0^{30}, 1^{60}, 2^{10}, 3^{10} \}$ \\
           & $8: \{ 0^{50}, 1^{20}, 2^{30}, 3^{10} \}$ \\
           & $9: \{ 0^{40}, 1^{30}, 2^{40} \}$ \\
        \hline 
        13 & $2: \{ 0^{48}, 1^{84}, 3^{24} \}$ \\
           & $3: \{ 0^{60}, 1^{48}, 2^{36}, 3^{12} \}$ \\
           & $4: \{ 0^{48}, 1^{72}, 2^{24}, 3^{12} \}$ \\
           & $5: \{ 0^{24}, 1^{108}, 2^{24} \}$ \\
           & $6: \{ 0^{60}, 1^{66}, 2^{12}, 3^{12}, 5^{6} \}$ \\
           & $7: \{ 0^{60}, 1^{66}, 2^{12}, 3^{12}, 5^{6} \}$ \\
           & $8: \{ 0^{24}, 1^{108}, 2^{24} \}$ \\
           & $9: \{ 0^{48}, 1^{72}, 2^{24}, 3^{12} \}$ \\
           & $10: \{ 0^{60}, 1^{48}, 2^{36}, 3^{12} \}$ \\
           & $11: \{ 0^{48}, 1^{84}, 3^{24} \}$ \\
        \hline
    \end{tabular}
    \caption{Properties of $F(x)=x^{\frac{p^n+1}{2}}+ux^2$ for different values of $(p, u)$ and $n=1$}
    \label{tab:differential_spectrum}
\end{table*}
 \begin{table*}[ht]
    \centering
    \renewcommand{\arraystretch}{1.3}
    \begin{tabular}{c l}
        \hline
        $p$ & $u$\ :\ Differential Spectrum \\
        \hline
17 & $2: \{ 0^{64}, 1^{144}, 2^{64} \}$\\
& $3: \{ 0^{96}, 1^{136}, 2^{16}, 4^{16}, 5^8 \}$\\
&$4: \{ 0^{96}, 1^{112}, 2^{32}, 3^{32} \}$\\
& $5: \{ 0^{80}, 1^{144}, 2^{16}, 3^{32} \}$\\
&$6: \{ 0^{112}, 1^{80}, 2^{48}, 3^{32} \}$\\
& $7: \{ 0^{64}, 1^{144}, 2^{64} \}$\\
& $8: \{ 0^{80}, 1^{112}, 2^{80} \}$\\
& $9: \{ 0^{80}, 1^{112}, 2^{80} \}$\\
& $10: \{ 0^{64}, 1^{144}, 2^{64} \}$\\
& $11: \{ 0^{112}, 1^{80}, 2^{48}, 3^{32} \}$\\
&  $12: \{ 0^{80}, 1^{144}, 2^{16}, 3^{32} \}$\\
& $13: \{ 0^{96}, 1^{112}, 2^{32}, 3^{32} \}$\\
& $14: \{ 0^{96}, 1^{136}, 2^{16}, 4^{16}, 5^{8} \}$\\
& $15: \{ 0^{64}, 1^{144}, 2^{64} \}$ \\    
        \hline
19  & $2: \{ 0^{108}, 1^{180}, 2^{18}, 3^{18}, 4^{18} \}$\\
& $3: \{  0^{108}, 1^{144}, 2^{72}, 3^{18} \}$\\
& $4: \{  0^{108}, 1^{126}, 2^{108} \}$\\
& $5: \{  0^{108}, 1^{180}, 2^{18}, 3^{18}, 4^{18} \}$\\
& $6: \{  0^{90}, 1^{198}, 2^{18}, 3^{36} \}$\\
& $7: \{  0^{126}, 1^{126}, 2^{72}, 4^{18} \}$\\
& $8: \{  0^{126}, 1^{126}, 2^{72}, 4^{18} \}$\\
& $9: \{  0^{72}, 1^{198}, 2^{72} \}$\\
& $10: \{  0^{72}, 1^{198}, 2^{72} \}$\\
& $11: \{  0^{126}, 1^{126}, 2^{72}, 4^{18} \}$\\
& $12: \{  0^{126}, 1^{126}, 2^{72}, 4^{18} \}$\\
& $13: \{  0^{90}, 1^{198}, 2^{18}, 3^{36} \}$\\
& $14: \{  0^{108}, 1^{180}, 2^{18}, 3^{18}, 4^{18} \}$\\
& $15: \{  0^{108}, 1^{126}, 2^{108} \}$\\
& $16: \{  0^{108}, 1^{144}, 2^{72}, 3^{18} \}$\\
& $17: \{  0^{108}, 1^{180}, 2^{18}, 3^{18}, 4^{18} \}$\\
\hline
    \end{tabular}
    \caption{Properties of $F(x)=x^{\frac{p^n+1}{2}}+ux^2$ for different values of $(p, u)$ and $n=1$}
    \label{tab:differential_spectrum2}
\end{table*}

For $p=5, n=2$, and $g$ a primitive root in the corresponding finite field (with the regular Magma~\cite{BCP97} primitive polynomial implementation), the  possible differential uniformity for various values of $u$ are $2,3,4,5$, and the function is APN for $u$ equal to $g^3,g^9,g^{15},g^{21}$. For $p=5, n=3$, the  possible differential uniformity for various values of $u$ are $2,3,4,5$, and the function is APN for $u$ equal to $2,3$. For $p=5, n=4$, the  possible differential uniformity for various values of $u$ are $3,4,5,6$; for $p=7,n=2$, the  possible differential uniformity for various values of $u$ are $2,3,4,5$,  and the function is APN for $u$ equal to $g^2,g^{12},g^{14},g^{26},g^{36},g^{38}$; for $p=7,n=3$, possible DU is $3,4,5$; for $p=11,13$ and $n=2$, possible DU is $3,4,5,6$.

Regarding the function $F(X)= x^{\frac{p^n-1}{2}+3}+ux^3$, $u\neq 0,\pm1$, for $p=5,n=1$, the function is APN for $u=1,2,3$ and has DU $3$ for the other values; for $p=5,n=2$, the function is APN for all values of $u$, except for $g^4,g^8,g^{16},g^{20}$, when it has DU  $9$; for $p=5,n=3$ the DU values are $3,4,6,7,8$; for $p=7,n=1$, the DU values are $2,3$, the function being APN for $u=2,3$; for $p=7,n=2$, the DU values are $4,6$; for $p=7,n=2$, the DU values are $ 3,4,5$; for $p=11,n=1$, the function is APN for $u$ equal to $3,8$, and for the remaining values of $u$, it has a DU of $3$; for $p=11,n=2$, the DU values are $4,5,6,8$; 
for $p=13,n=1$, the DU values are $2,3$, the function being APN for $u=2,4,6$, the function being APN for $u=2,11$. Some sample differential spectrum for this function is displayed below.
 \begin{table*}[ht]
    \centering
    \renewcommand{\arraystretch}{1.3}
    \begin{tabular}{c c l}
        \hline
        $p$ & $n$ & $u$\ :\ Differential Spectrum \\
        \hline
    5 & 1 & $2: \{0^{8}, 1^{4}, 2^{8}\}$   \\
   & & $3: \{0^{8}, 1^{4}, 2^{8}\}$  \\
   & &  $4: \{0^{8}, 1^{4}, 2^{8}\}$  \\
   \hline
    5 & 2 & $w: \{0^{276}, 1^{23}, 2^{276}\}$ \\
    && $w^2: \{0^{276}, 1^{23}, 2^{276}\} $ \\
    && $w^3\ \{0^{276}, 1^{23}, 2^{276}\}$ \\
    && $w^4: \{0^{460}, 1^{12}, 4^{44}, 6^{48}, 9^{11}\}$ \\
   & & $w^5: \{0^{276}, 1^{23}, 2^{276}\}$\\
   & & $2: \{0^{276}, 1^{23}, 2^{276}\}$ \\
    && $w^7: \{0^{276}, 1^{23}, 2^{276}\} $ \\
    && $w^8: \{0^{460}, 1^{11}, 4^{48}, 6^{44}, 9^{12}\} $ \\
    && $w^9: \{0^{276}, 1^{23}, 2^{276}\}  $\\
    && $w^{10}: \{0^{276}, 1^{23}, 2^{276}\} $ \\
    && $w^{11}: \{0^{276}, 1^{23}, 2^{276}\}$ \\
    && $4: \{0^{460}, 1^{12}, 4^{44}, 6^{48}, 9^{11}\}$ \\
    && $ w^{13}: \{0^{276}, 1^{23}, 2^{276}\}$ \\
    && $w^{14}\ \{0^{276}, 1^{23}, 2^{276}\} $ \\
    & & $ w^{15}: \{0^{276}, 1^{23}, 2^{276}\} $\\
    && $w^{16}: \{0^{460}, 1^{11}, 4^{48}, 6^{44}, 9^{12}\} $\\
    && $w^{17}\ \{0^{276}, 1^{23}, 2^{276}\} $ \\
    && $3: \{0^{276}, 1^{23}, 2^{276}\}  $ \\
     && $w^{19}: \{0^{276}, 1^{23}, 2^{276}\} $\\
     && $w^{20}: \{0^{460}, 1^{12}, 4^{44}, 6^{48}, 9^{11}\} $ \\
     && $w^{21}: \{0^{276}, 1^{23}, 2^{276}\}$ \\
     && $w^{22}: \{0^{276}, 1^{23}, 2^{276}\}$\\
     && $w^{23}: \{0^{276}, 1^{23}, 2^{276}\}$ \\
    \hline
    \end{tabular}
    \caption{Properties of $F(x)=x^{\frac{p^n+1}{3}}+ux^3$ for different values of $(p,n, u)$}
    \label{tab:differential_spectrum3}
\end{table*}

For the function $F(x)=x^{\frac{p^n-1}{2}+p^b+1}+x^{p^a+1}$, $0\leq a\leq b<n$ ($(a,b)=(0,0)$ was previously considered, so we avoid it here), for $(p,n,a,b)=(5, 2, 0, 1)$, possible DU values are $2,3,5$; for $(p,n,a,b)=(5, 2, 1, 1)$, possible DU values are $5,9,13$; for $p,n,a,b)=(5, 3, 0, 1)$, possible DU values are $6,7$; for $(p,n,a,b)=(5, 3, 0, 2)$, possible DU values are $7$; for $(p,n,a,b)=(5, 3, 1, 2)$, possible DU values are $6,7$; for $(p,n,a,b)=(5, 3, 2,2)$, possible DU values are $2,3,4,5$; for $(p,n,a,b)=(7, 2, 0, 1)$, possible DU values are $4,5,7$.

\end{document}